\documentclass{amsart}
\usepackage{amssymb,amsthm,amsmath,epic,eepic,ulem}
\usepackage{alltt}

\input kmacros3.sty
\input xy
\xyoption{all}

\newcommand{\stm}{sharp test multiplier}
\newcommand{\ste}{sharp test element}
\newcommand{\sfp}{sharply $F$-pure}
\newcommand{\sfpty}{sharp $F$-purity}
\newcommand{\sfc}{sharp Frobenius closure}
\begin{document}

\title{Generalized test ideals, sharp $F$-purity, and sharp test elements}
\author{Karl Schwede}
\begin{abstract}
Consider a pair $(R, \ba^t)$ where $R$ is a ring of positive characteristic, $\ba$ is an ideal such that $\ba \cap R^{\circ} \neq \emptyset$, and $t > 0$ is a real number.  In this situation we have the ideal $\tau_R(\ba^t)$, the generalized test ideal associated to $(R, \ba^t)$ as defined by Hara and Yoshida.  We show that $\tau_R(\ba^t) \cap R^{\circ}$ is made up of appropriately defined generalized test elements which we call \emph{\ste s}.    We also define a variant of $F$-purity for pairs, \emph{\sfpty}, which interacts well with \ste s and agrees with previously defined notions of $F$-purity in many common situations.  We show that if $(R, \ba^t)$ is \sfp, then $\tau_R(\ba^t)$ is a radical ideal.  Furthermore, by following an argument of Vassilev, we show that if $R$ is a quotient of an $F$-finite regular local ring and $(R, \ba^t)$ is \sfp, then $R/\tau_R(\ba^t)$ itself is $F$-pure.  We conclude by showing that \sfpty{ }can be used to define the $F$-pure threshold.  As an application we show that the $F$-pure threshold must be a rational number under certain hypotheses.
\end{abstract}
\thanks{The author was partially supported by a National Science Foundation postdoctoral fellowship.}
\address{Department of Mathematics\\ University of Michigan\\ East Hall,
530 Church Street \\ Ann Arbor, Michigan, 48109}
\email{kschwede@umich.edu}
\maketitle

\section{Introduction}

In this note, we consider the various aspects of the generalized test ideal $\tau_R(\ba^t)$ associated to the pair $(R, \ba^t)$, as defined by Hara and Yoshida \cite{HaraYoshidaGeneralizationOfTightClosure}; also see \cite{TakagiInterpretationOfMultiplierIdeals}.  Here $R$ is a ring of positive characteristic $p$, $\ba$ is an ideal such that $\ba \cap R^{\circ} \neq \emptyset$ and $t > 0$ is a real number.  The ideal $\tau_R(\ba^t)$ is a positive characteristic analogue of the multiplier ideal (see \cite{SmithMultiplierTestIdeals} and \cite{HaraYoshidaGeneralizationOfTightClosure}) which has been heavily studied in recent years.  The reason for the name \emph{test ideal} is that in the classical case where $\ba = R$, $\tau(R^1) \cap R^{\circ}$ is equal to the set of tight closure test elements (elements of $R^{\circ}$ which can be used to ``test'' tight closure containments; see \cite{HochsterHunekeTC1}).

However, when tight closure for pairs $(R, \ba^t)$ was defined, see \cite{HaraYoshidaGeneralizationOfTightClosure} and also  \cite{TakagiInterpretationOfMultiplierIdeals}, a generalization of test elements, called \emph{$\ba^t$-test elements}, was introduced and studied.  In particular, it was easy to see that  $\tau_R(\ba^t) \cap R^{\circ}$ is not equal to the set of these test elements even if $R$ is regular, $\ba$ is principal and $t = 1$.  We introduce a new type of generalized test element which we call a \emph{\ste}.  An important point is that these \ste s still ``test'' tight closure containments.  Furthermore, $\tau_R(\ba^t) \cap R^{\circ}$ is in fact equal to the set of \ste s.  Thus our main result is the following.

\vskip 12pt
\hskip-12pt
{\bf Corollary \ref{TestIdealsAreSharpTestElements}}
\hskip 3pt {\it The set of  $\ba^t$-\ste s is equal to $\tau(\ba^t) \cap R^{\circ}$.}
\vskip 12pt


This definition of generalized test elements also suggests a redefinition of $F$-purity for pairs $(R, \ba^t)$, which we call \emph{\sfpty}; compare with \cite{HaraWatanabeFRegFPure} and \cite{TakagiInversion}.  Once we have this definition, we are able to generalize several results from the classical case ($\ba = R$), to the case of a pair $(R, \ba^t)$.  In particular, simply following classical arguments, we are able to show that if $(R, \ba^t)$ is \sfp, then $\tau_R(\ba^t)$ is a radical ideal; see Corollary \ref{SharpFPureImpliesRadicalTestIdeal}.  Furthermore, under certain mild hypotheses on $R$, we are able to obtain the following generalization of a theorem of Vassilev; see \cite{VassilevTestIdeals}.

\vskip 12pt
\hskip-12pt
{\bf Corollary \ref{GeneralizedVassilevSubadjunction}}
\hskip 3pt {\it Suppose $R$ is a quotient of an $F$-finite regular ring, $\ba \subseteq R$ an ideal such that $\ba \cap R^{\circ} \neq \emptyset$ and $t > 0$ a real number such that $(R, \ba^t)$ is \sfp. Then $R / \tau_R(\ba^t)$ is $F$-pure.}
\vskip12pt

\hskip-12ptThis corollary should be thought of as related to \cite[Theorem 1]{KawamataSubadjunction2}, \cite[Proposition 3.2]{AmbroSeminormalLocus}, and \cite[Theorem 5.5]{SchwedeEasyCharacterization}, all results in characteristic zero which use very different techniques.

Finally, we also show that \sfpty{ }can be used to define $\fpt (\ba)$, the  $F$-pure threshold of $(R, \ba^*)$; see \cite{TakagiWatanabeFPureThresh}.  As an easy corollary we are able to show that if $(R, \ba^{\fpt(\ba)})$ is \sfp, then $\fpt(\ba)$ is a rational number that can be written as a quotient whose denominator is not a multiple of $p = \Char R$; see Corollary \ref{CorollaryFPureThresholdIsRational}.

It should be noted that many of the arguments presented in this paper are simply generalizations of arguments from the ``classical case''.  However, this strongly suggests that the \emph{\ste} which we introduce is the correct notion of a generalized test element and that \emph{\sfpty} is the right notion of $F$-purity for pairs.
\\
\\
{\it Acknowledgements:}
\\
The author would like to thank Mircea Musta{\c{t}}{\v{a}} and Mel Hochster for several valuable discussions, thank Shunsuke Takagi for pointing out several typos in a previous draft, and also thank Mel Hochster for a careful reading of this paper.  Finally, the author would like to thank the referee for several valuable suggestions as well as pointing out several typos.

\section{Sharp Test Elements}

All rings will be assumed to be Noetherian excellent reduced and containing a field of characteristic $p > 0$.  Given a ring $R$, we define $R^{\circ}$ to be the set of all elements of $R$ not contained in any minimal prime of $R$.  If $\ba$ is an ideal of $R$ such that $\ba \cap R \neq \emptyset$, then we define $\ba^0 = R$.  If $I$ is an ideal of $R$, then we define $I^{[p^e]}$ to be the ideal generated by the $p^e$th powers of elements of $I$.

\begin{definition} \cite{HaraYoshidaGeneralizationOfTightClosure} \cite{TakagiInterpretationOfMultiplierIdeals} \cite{HochsterHunekeTC1}
\label{DefinitionTestIdeal}
Let $R$ be a reduced ring and let $I$ be an ideal of $R$.  Further suppose that $\ba$ is an ideal such that $\ba \cap R^{\circ} \neq \emptyset$ and $t > 0$ is a real number.  We say that $z$ is in the \emph{$\ba^t$-tight closure of $I$}, denoted $I^{*\ba^t}$, if there exists a $c \in R^{\circ}$ such that
\[
c \ba^{\lceil t p^e \rceil} z^{p^e} \subseteq I^{[p^e]}
\]
for all $e \gg 0$.
\end{definition}

The next two lemmas show that we can alter certain conditions in the definition of $I^{* \ba^t}$ without changing $I^{*\ba^t}$.

\begin{lemma}  Assume the notation of Definition \ref{DefinitionTestIdeal}.  Then $z \in I^{*\ba^t}$ if and only if there exists some $c \in R^{\circ}$ such that
\[
c \ba^{\lceil t p^e \rceil} z^{p^e} \subseteq I^{[p^e]}
\]
for all $e \geq 0$.
\end{lemma}
\begin{proof}
We follow the main idea of the proof of \cite[Proposition 4.1(c)]{HochsterHunekeTC1}.  Suppose that there exists a $c'$ such that $c' \ba^{\lceil t p^e \rceil} z^{p^e} \subseteq I^{[p^e]}$ for all $e > e_0$.  Since $R$ is reduced, we note that $S = (R^{\circ})^{-1} R$ is a product of fields, and thus every ideal in $S$ is radical.  Therefore $I^{[q]} S = IS$ for all $q = p^e$.  Also note that since $\ba \cap R^{\circ} \neq \emptyset$ we have $\ba^{\lceil tq \rceil} S = S$.  But now if $c' \ba^{\lceil t p^e \rceil} z^{p^e} \subseteq I^{[p^e]}$ for $e > e_0$, we see that $z \in IS$ and furthermore that $z^{p^e} \in (IS)^{[p^e]} = I^{[p^e]} S$ for all $e$, and in particular, for each $e \leq e_0$.  Thus we can pick $c_e \in R^{\circ}$ so that $c_e z^{p^e} \in I^{[p^e]}$ for each $e \leq e_0$.  We set $c = c' \prod_{i = 1, \ldots, e_0} c_i$.  It follows easily that $c \ba^{\lceil t p^e \rceil} z^{p^e} \subseteq I^{[p^e]}$ for all $e \geq 0$.
\end{proof}

\begin{lemma}
\label{RightCharacterizationOfTightClosure}
Assume the notation of Definition \ref{DefinitionTestIdeal}.  Then $z \in I^{*\ba^t}$ if and only if there exists some $c \in R^{\circ}$ such that
\[
c \ba^{\lceil t (p^e - 1) \rceil} z^{p^e} \subseteq I^{[p^e]}.
\]
for all $e \geq 0$.
\end{lemma}
\begin{proof}
The idea of the proof is essentially  the same as in \cite[Proposition 2.2]{HaraWatanabeFRegFPure}.
 Suppose that there exists a $c \in R^{\circ}$ such that $c \ba^{\lceil t p^e \rceil} z^{p^e} \in I^{[p^e]}$ for all $e \geq 0$.  Let $n$ be an exponent such that $\ba^n \ba^{\lceil t (p^e - 1) \rceil} \subseteq (\ba^{\lceil t p^e \rceil})$ for all $e \geq 0$ (for example, $n = \lceil t + 1 \rceil$ works).  Choose $f \in \ba \cap R^{\circ}$.  Then note that $(c f^n) \ba^{\lceil t (p^e - 1) \rceil} z^{p^e} \subseteq (c \ba^n) \ba^{\lceil t (p^e - 1) \rceil} z^{p^e} \subseteq I^{[p^e]}$ for all $e \geq 0$.
\end{proof}

It is the characterization of $I^{*\ba^t}$ from Lemma \ref{RightCharacterizationOfTightClosure} that we will make use of for the remainder of the paper.


We now remind the reader of the definition of the generalized test ideal.

\begin{definition}  \cite{HaraYoshidaGeneralizationOfTightClosure}
Assume the notation of \ref{DefinitionTestIdeal}.  Recall that the \emph{generalized test ideal of $\ba^t$}, denoted by $\tau(\ba^t)$ is simply defined to be
\[
\bigcap_{J \subseteq R} (J : J^{*\ba^t}) \text{ where $J$ runs through all ideals of $R$.}
\]
\end{definition}

We next define a new type of test element, the \emph{\ste}.

\begin{definition}
We define $c \in R$ to be an \emph{$\ba^t$-\stm} for $R$ if, for all ideals $I$ and all $z \in I^{*\ba^t}$, we have $c \ba^{\lceil t (p^e - 1) \rceil} z^{p^e} \subseteq I^{[p^e]}$ for all $e \geq 0$.
We say that $c \in R^{\circ}$ is an \emph{$\ba^t$-\ste} if $c$ is also an $\ba^t$-\stm.
\end{definition}

Notice first that if $c$ is an $\ba^t$-\stm, then $c$ is in $\tau(\ba^t)$ because we can always consider the situation where $e = 0$.  We will show that the reverse inclusion also holds.  First we need a lemma.

\begin{lemma}
\label{LemmaPowerIsInTightClosure}
Suppose that $z \in I^{*\ba^t}$.  Then for every power $q = p^e$, $e \geq 0$ we have $\ba^{\lceil t(q - 1) \rceil} z^{q} \subseteq (I^{[q]})^{*\ba^t}$.
\end{lemma}
\begin{proof}
By assumption, we know that there exists a $c \in R^{\circ}$ such that $c \ba^{\lceil t(p^d - 1) \rceil} z^{p^d} \subseteq I^{[p^d]}$ for all $d \geq 0$.  For every $f \in \ba^{\lceil t(q - 1) \rceil}$, we wish to show that there exists a $c' \in R^{\circ}$ such that for all $d \geq 0$,
\[
c' \ba^{\lceil t(p^d - 1) \rceil} (f z^{q})^{p^d} \subseteq (I^{[q]})^{[p^d]} = I^{[p^{e+d}]}.
\]
But note that $\ba^{\lceil t(p^d - 1) \rceil} (f z^{q})^{p^d} \subseteq \ba^{\lceil t(p^d - 1) \rceil} (\ba^{\lceil t(q - 1) \rceil} z^{q})^{p^d}$ and so it is enough to show that
\[
c' \ba^{\lceil t(p^d - 1) \rceil} (\ba^{\lceil t(q - 1) \rceil} z^{q})^{p^d} \subseteq (I^{[q]})^{[p^d]} = I^{[p^{e+d}]}.
\]
Now we observe that $\ba^{\lceil t(p^d - 1) \rceil} (\ba^{\lceil t(q - 1) \rceil} z^{q})^{p^d} = \ba^{(\lceil t(p^d - 1) \rceil + p^d \lceil t(p^e - 1) \rceil)} z^{p^{e+d}}$.
Set $c' = c$, and we see it is sufficient to check that $\lceil t(p^d - 1) \rceil + p^d \lceil t(p^e - 1) \rceil \geq \lceil t(p^{e+d} - 1) \rceil$ for all $d \geq 0$.  However, it is easy to see that
\[
\lceil t(p^d - 1) \rceil + p^d \lceil t(p^e - 1) \rceil \geq t(p^d - 1) + p^d t(p^e - 1) = t(p^d - 1 + p^{d+e}-p^d) = t (p^{d+e} - 1)
\]
But the left side is an integer, and our needed inequality is proven.
\end{proof}

\begin{theorem}
\label{TestIdealsAreSharpTestMultipliers}
The set of  $\ba^t$-\stm s  is equal to $\tau(\ba^t)$.
\end{theorem}
\begin{proof}
Suppose that $c \in \tau(\ba^t)$ and suppose $z \in I^{*\ba^t}$.  But then, using Lemma \ref{LemmaPowerIsInTightClosure}, for every $e \geq 0$ with $q = p^e$ we have,
\[
\ba^{\lceil t(q - 1) \rceil} z^{q} \subseteq (I^{[q]})^{*\ba^t}
\]
and so
\[
c \ba^{\lceil t(q - 1) \rceil} z^{q} \subseteq I^{[q]}
\]
which proves that $c$ is an $\ba^t$-\stm.
\end{proof}

\begin{corollary}
\label{TestIdealsAreSharpTestElements}
The set of  $\ba^t$-\ste s is equal to $\tau(\ba^t) \cap R^{\circ}$.
\end{corollary}

We can also show that \ste s exist.

\begin{theorem}
Suppose we are given a reduced $F$-finite ring $R$, an ideal $\ba \subseteq R$ such that $\ba \cap R^{\circ} \neq \emptyset$ and a real number $t > 0$.  Then $\ba^t$-\ste s exist; in other words, $\tau(\ba^t) \cap R^{\circ} \neq \emptyset$.
\end{theorem}
\begin{proof}
By hypothesis, $R$ is $F$-finite, reduced and $\ba \cap R^{\circ} \neq \emptyset$.  Thus it follows from \cite[Theorem 6.4]{HaraYoshidaGeneralizationOfTightClosure} that there exists some element $c \in R^{\circ}$ such that for all ideals $I$, all $z \in I^{* \ba^t}$ and all $e \geq 0$, we have $c \ba^{\lceil tp^e \rceil} z^{p^e} \subseteq I^{[p^e]}$.  Choose an element $d \in R^{\circ}$ such that $d \ba^{\lceil t(p^e - 1) \rceil} \subseteq \ba^{\lceil t p^e \rceil}$ for all $e \geq 0$ (note that such an element exists since $\ba \cap R^{\circ} \neq \emptyset$).  Then $c' = cd \in R^{\circ}$ is a sharp test element since
\[
c' \ba^{\lceil t(p^e - 1) \rceil} z^{p^e} = c d \ba^{\lceil t(p^e - 1) \rceil} z^{p^e} \subseteq c \ba^{\lceil tp^e \rceil} z^{p^e} \subseteq I^{[p^e]}.
\]
\end{proof}

See \cite[Section 2]{HaraTakagiOnAGeneralizationOfTestIdeals} for a comparison of a different sort of test element with elements of the generalized test ideal.



\section{\sfpty}

In this section we define and study a variant of $F$-purity for pairs that works well with these \ste s.  We use the following notation throughout this section. Let $R$ be an $F$-finite reduced ring of positive characteristic $p > 0$.  Further suppose that $\ba$ is an ideal with $\ba \cap R^{\circ} \neq \emptyset$ and $t > 0$ is a real number.  If $\ba = (f)$ is a principal ideal, then we will write $(R, f^t)$ for the pair $(R, \ba^t)$.

Recall that a pair $(R, \ba^t)$ is said to be \emph{$F$-pure} if for every $e \gg 0$ there exists a $f \in \ba^{\lfloor t(p^e-1) \rfloor}$, such that the map
\[
f^{1 \over p^e} F^e : R \rightarrow R^{1\over p^e},
\]
that sends $1$ to $f^{1 \over p^e}$, splits.  Recall that $R$ itself is said to be $F$-pure if $(R, R^1)$ is $F$-pure.  See \cite[Definition 3.1]{TakagiInversion} and compare with \cite[Section 2]{HaraWatanabeFRegFPure} and \cite{HochsterRobertsFrobeniusLocalCohomology}.  Also recall that a pair $(R, \ba^t)$ is said to be \emph{strongly $F$-pure} if there exists $e > 0$ (equivalently for all $e \gg 0$) and $f \in \ba^{\lceil t p^e \rceil}$ such that the map $f^{1 \over p^e} F^e$ splits; see \cite{TakagiWatanabeFPureThresh}.  Note that a strongly $F$-pure pair is $F$-pure; see \cite[Proposition 1.5]{TakagiWatanabeFPureThresh}.

\begin{definition}
We say that $R$ is \emph{$\ba^t$-\sfp}, or that the pair $(R, \ba^t)$ is \emph{\sfp}, if there exist infinitely many integers $e \geq 0$ and associated elements $f_e \in \ba^{\lceil t(p^e-1) \rceil}$ such that each map
\[
f_e^{1 \over p^e} F^e : R \rightarrow R^{1\over p^e},
\]
which sends $1$ to $f_e^{1 \over p^e}$, splits.
\end{definition}

\begin{remark}
\label{RemarkFPureImpliesSharplyFPure}
It is easy to see that if $(R, \ba^t)$ is $F$-pure in the usual sense and $(p^e - 1) t$ is an integer for infinitely many $e > 0$ (equivalently, $(p^e - 1) t$ is an integer for some $e > 0$, equivalently $t$ is a rational number which can be written with an integer denominator which is not a multiple of $p$), then $R$ is $\ba^t$-\sfp.  Furthermore, if $(R, \ba^t)$ is strongly $F$-pure, it is clearly both sharply $F$-pure and $F$-pure.
\end{remark}

It turns out that one only needs to exhibit a single splitting in order to prove that a pair is \sfp, as the following proposition shows.

\begin{proposition}
\label{OneSplittingImpliesAllSplitting}
Suppose that there exists a single $e > 0$ and an element $f \in \ba^{\lceil t(p^e-1) \rceil}$ such that the map
$f^{1 \over p^e} F^e : R \rightarrow R^{1\over p^e}$ splits.  Then $(R, \ba^t)$ is \sfp.
\end{proposition}
\begin{proof}
It is sufficient to show that for all $d = ne$, $n$ a positive integer, there exists an element $g \in \ba^{\lceil t(p^d-1) \rceil}$ such that $g^{1 \over p^d} F^d : R \rightarrow R^{1\over p^d}$ splits.  By using the identification of $R^{1 \over p^e}$ with $R$ (since $R$ is reduced) we compose the map $f^{1 \over p^e} F^e $ with itself $n$-times.  This is a map from $R$ to $R^{1 \over p^{ne}}$ which sends $1$ to $f^{1 + p^e + p^{2e} + \dots + p^{(n-1)e} \over p^{ne}}$, and this map splits by construction.  Thus, it is sufficient to show that $f^{1 +p^e + \dots + p^{(n-1)e}} \in \ba^{\lceil t(p^{ne} - 1) \rceil}$.  However, we do know that
\[
f^{1+ \dots + p^{(n-1)e}} \in \ba^{(1+ \dots + p^{(n-1)e})\lceil t(p^e-1) \rceil}.
\]
Thus it is sufficient to show that
\[
(1+ \dots + p^{(n-1)e})\lceil t(p^e-1) \rceil \geq \lceil t(p^{ne} - 1) \rceil.
\]
But we do know that
\[
(1+ \dots + p^{(n-1)e})\lceil t(p^e-1) \rceil \geq (1+ \dots + p^{(n-1)e}) t(p^e-1) =  t(p^{ne}-1)
\]
Thus, since the left side is an integer, we are done.
\end{proof}

\begin{corollary}
\label{SharplyFPureImpliesSharplyFPureElement}
If there is an $f \in R$ such that the map $f^{1 \over p^e} F^e : R \rightarrow R^{1\over p^e}$ splits, then $(R, f^{1 \over p^e - 1})$ is \sfp.
\end{corollary}

We now show that if we are working with principal ideals, then $F$-purity and \sfpty{ }agree much of the time.

\begin{proposition}
\label{PropositionInfiniteSplittingImpliesAllSplitting}
If $(R, f^t)$ is \sfp, then $(R, f^t)$ is $F$-pure (and in fact the map for determining $F$-purity splits at every $e \geq 0$).
\end{proposition}
\begin{proof}
Choose $d > 0$ and consider the map
\[
(f^{\lfloor t(p^d-1) \rfloor})^{1 \over p^d} F^d : R \rightarrow R^{1\over p^d}
\]
which sends $1$ to $(f^{\lfloor t(p^d-1) \rfloor})^{1 \over p^d}$.  Choose $e > d$ such that
\[
(f^{\lceil t(p^e-1) \rceil})^{1 \over p^e} F^e : R \rightarrow R^{1\over p^e}
\]
splits.
Since $R$ itself is $F$-split, we can compose the map $(f^{\lfloor t(p^d-1) \rfloor})^{1 \over p^d} F^d$ with $F^{e - d}$.
This gives us the map $R \rightarrow R^{1 \over p^e}$ that sends $1$ to $(f^{\lfloor t(p^d-1) \rfloor})^{p^{e-d} \over p^e} = (f^{p^{e-d}\lfloor t(p^d-1) \rfloor})^{1 \over p^e}$.
Thus it is sufficient to show that $p^{e-d}\lfloor t(p^d-1) \rfloor \leq \lceil t(p^e-1) \rceil$.  Of course,
\[
p^{e-d}\lfloor t(p^d-1) \rfloor \leq p^{e-d}(t(p^d - 1)) = t(p^e - p^{e-d}) \leq t(p^e - 1) \leq \lceil t(p^e-1) \rceil
\]
as desired.
\end{proof}

\begin{corollary}
\label{SharplyFPureImpliesFPureElement}
If there is an $f \in R$ such that the map $f^{1 \over p^e} F^e : R \rightarrow R^{1\over p^e}$ splits, then $(R, f^{1 \over p^e - 1})$ is $F$-pure.
\end{corollary}


\begin{corollary}
The pair $(R, \ba^t)$ is \sfp{ }if and only if there exists some $e > 0$ and $f \in \ba^{\lceil t(p^e-1) \rceil}$ such that $(R, f^{1 \over p^e - 1})$ is $F$-pure (equivalently, $(R, f^{1 \over p^e - 1})$ is \sfp).
\end{corollary}
\begin{proof}
If $(R, \ba^t)$ is \sfp, then it follows by Corollary \ref{SharplyFPureImpliesSharplyFPureElement} that there exists some $e$ and $f$ as above such that $(R, f^{1 \over p^e - 1})$ is \sfp, and so it follows, by Proposition \ref{PropositionInfiniteSplittingImpliesAllSplitting}, that $(R, f^{1 \over p^e - 1})$ is $F$-pure.

Conversely, if there exists an $e$ and $f$ as above such that $(R, f^{1 \over p^e - 1})$ is $F$-pure, then the pair is also \sfp{ }by Remark \ref{RemarkFPureImpliesSharplyFPure}.  Furthermore, by Proposition \ref{PropositionInfiniteSplittingImpliesAllSplitting}, the map
\[
(f^{\lfloor {1 \over p^e - 1}(p^e - 1) \rfloor})^{1 \over p^e} F^e : R \rightarrow R^{1 \over p^e}
\]
splits.  But $f^{\lfloor {1 \over p^e - 1}(p^e - 1) \rfloor} = f \in \ba^{\lceil t(p^e-1) \rceil}$ by assumption, and so the corollary is proven.
\end{proof}

The previous results suggest the following simple question.

\begin{question} If $(R, \ba^t)$ is \sfp, then is $(R, \ba^t)$ $F$-pure? \end{question}

\begin{remark}
In Section \ref{SectionApplicationToFPureThresholds}, we show that if $(R, \ba^t)$ is sharply $F$-pure, then $(R, \ba^{t - \epsilon})$ is strongly $F$-pure (and thus $F$-pure) for every $\epsilon$ satisfying $t \geq \epsilon > 0$.
\end{remark}

\begin{definition}
\label{DefinitionSharpFrobeniusClosure}
Given an ideal $I \subseteq R$, we define the \emph{$\ba^t$-\sfc{ }of $I$}, denoted $I^{F^{\sharp} \ba^t}$, as follows.  The ideal $I^{F^{\sharp} \ba^t}$ is the set of all elements $z$ such that $\ba^{\lceil t (p^e - 1) \rceil} z^{p^e} \subseteq I^{[p^e]}$ for all $e \gg 0$.
\end{definition}

\begin{remark}
Note that if $(R, \ba^t)$ is \sfp, then $I^{F^{\sharp} \ba^t} = I$ for all ideals $I$.  The argument is essentially the same as the argument that a strongly $F$-regular ring is weakly $F$-regular; see \cite[Proposition 3.1]{HochsterHunekeTightClosureAndStrongFRegularity}.
\end{remark}

\begin{remark}
In Definition \ref{DefinitionSharpFrobeniusClosure}, it might seem natural to replace the condition ``for all $e \gg 0$'' with the condition ``for infinitely many $e > 0$''.  We will be working in the case of an \sfp{ }pair.  Therefore, if for infinitely many $e > 0$, we have the containment $\ba^{\lceil t (p^e - 1) \rceil} z^{p^e} \subseteq I^{[p^e]}$, it is not clear whether we still have such a containment at an $e > 0$ where the map used to define \sfpty{ }splits.  Furthermore, the condition that $\ba^{\lceil t (p^e - 1) \rceil} z^{p^e} \subseteq I^{[p^e]}$ for all $e \gg 0$ arises naturally in Lemma \ref{EasyComputationOfSharpFrobeniusClosure}.
\end{remark}

The following lemma can be used to check containment in common situations.

\begin{lemma}
\label{EasyComputationOfSharpFrobeniusClosure}
In the setup of the previous definition, suppose that $\ba = (f)$ and there exists some $e > 0$ such that $\ba^{\lceil t (p^e - 1) \rceil } z^{p^e} \subseteq I^{[p^e]}$ and so that $t(p^e - 1)$ is an integer, then $z \in I^{F^{\sharp} \ba^t}$
\end{lemma}
\begin{proof}
Suppose that $f^{\lceil t (p^e - 1) \rceil } z^{p^e} \subseteq I^{[p^e]}$.  Then, for any integer $d > 0$ we have
\[
f^{p^d \lceil t (p^e - 1) \rceil } z^{p^{e+d}} \subseteq I^{[p^{e+d}]}.
\]
Therefore, it suffices to show that $\lceil t (p^{d + e} - 1) \rceil \geq p^d \lceil t (p^e - 1) \rceil$.  Of course $t (p^e - 1) =  \lceil t (p^e - 1) \rceil$ by assumption so that $p^d \lceil t (p^e - 1) \rceil = p^d t (p^e - 1) = p^{d+e} t - t p^d$.  Now, note that $\lceil t (p^{d + e} - 1) \rceil \geq t (p^{d + e} - 1) \geq p^{d+e} t - t p^d$.
\end{proof}

Our last goal of the section is to prove that sharply $F$-pure pairs have radical test ideals.

\begin{proposition}
Suppose $c$ is an element of $R$ such that some power of $c$ is contained in $\tau(\ba^t)$.
Then for all ideals $I$, we have $c I^{*\ba^t} \subseteq I^{F^{\sharp} \ba^t}$.
\end{proposition}
\begin{proof}
Choose $z \in I^{*\ba^t}$.  Consider all $e \geq e_0$ such that $c^{p^e}$ is in $\tau(\ba^t)$.
But then we know that $\ba^{\lceil t (p^e - 1) \rceil } (cz)^{p^e} = c^{p^e} \ba^{\lceil t (p^e - 1) \rceil } z^{p^e} \subseteq I^{[p^e]}$ for all $e \geq e_0$.   This implies that $cz \in I^{F^{\sharp} \ba^t}$.
\end{proof}

\begin{corollary}
\label{SharpFPureImpliesRadicalTestIdeal}
Suppose that $(R, \ba^t)$ is sharply $F$-pure.  Then $\tau(\ba^t)$ is a radical ideal.
\end{corollary}
\begin{proof}
Suppose that $c^n \in \tau(\ba^t)$ for some $n > 0$.  Then we know that for any ideal $I$, $c I^{*\ba^t} \subseteq I^{F^{\sharp} \ba^t} = I$ where the final equality comes because $(R, \ba^t)$ is sharply $F$-pure.  Thus $c \in \tau(\ba^t)$ as desired.
\end{proof}

\begin{corollary}
Suppose that $(R, \ba^t)$ is $F$-pure and that $(p^e - 1)t$ is an integer for some (equivalently infinitely many) $e > 0$.  Then $\tau(\ba^t)$ is a radical ideal.
\end{corollary}

If instead of \sfpty{ }one considers the standard definition ($F$-purity), then Yoshida and Musta{\c{t}}{\u{a}} have constructed examples of pairs $(R, f^t)$ that are $F$-pure, but which have non-radical test-ideals; see \cite[Example 4.5]{MustataYoshidaTestIdealVsMultiplierIdeals}.

\section{On a generalization of a result of Vassilev's}

The first goal of this section is to show that a Fedder-type criterion extends to the context of \sfp{ }pairs.  As an application we show that the main result of \cite{VassilevTestIdeals}, also extends to the context of sharply $F$-pure pairs.  Her result, and its generalization below, can be thought of as closely related to certain subadjunction-type results in characteristic zero such as \cite{AmbroSeminormalLocus}, \cite{KawamataSubadjunction2} and \cite[Theorem 5.5]{SchwedeEasyCharacterization}.

\begin{theorem} \cite[Proposition 1.7]{FedderFPureRat}, \cite[Proposition 2.6]{HaraWatanabeFRegFPure}, \cite[Lemmas 3.4, 3.9]{TakagiInversion}
\label{SharpFeddersCriterion}
Suppose $(S, \bm)$ is an $F$-finite regular local ring, $I \subseteq S$ a radical ideal, $R = S/I$, $\ba'$ an ideal of $S$ containing $I$, with image in $R$ denoted by $\ba$ such that $\ba \cap R^{\circ} \neq \emptyset$, and $t > 0$ is a real number.  Then the following are equivalent.
\begin{itemize}
\item[(i)] $(R, \ba^t)$ is \sfp,
\item[(ii)]  For infinitely many $q = p^e$, there exists $d \in \ba^{\lceil t(q-1) \rceil}$ such that $d^{1 \over p^e} F^e : E_R \rightarrow E_R \tensor_R R^{1 \over q}$ is injective,
\item[(iii)]  $\ba'^{\lceil t (q - 1) \rceil} (I^{[q]} : I) \nsubseteq \bm^{[q]}$ for infinitely many $q = p^e$.
\end{itemize}
\end{theorem}
\begin{proof}
\numberwithin{equation}{theorem}
We begin by proving the equivalence of (i) and (ii).
Note that for any $d \in R$, the map $d^{1 \over q} F^e : R \rightarrow R^{1 \over q}$ splits if and only if the composition
\begin{equation}
\label{compositionSurjective}
\xymatrix{
\Hom_R(R^{1 \over q}, R) \ar[r]^-{d^{1 \over q}} & \Hom_R(R^{1 \over q}, R) \rightarrow \Hom_R(R, R)
}
\end{equation}
is surjective.  This surjection is clearly invariant under completion, so we may assume that $R$ is complete.
Following \cite[Lemma 3.4]{TakagiInversion}, we wish to show that the Matlis dual of $\Hom_R(R^{1 \over q}, R)$ is isomorphic to $E_R \tensor_R R^{1 \over q}$.
But note that, by Matlis duality, $\Hom_R(R^{1 \over q}, R) \cong \Hom_R(\Hom_R(R, E_R), \Hom_R(R^{1 \over q}, E_R))$.  On the other hand, the right side is naturally isomorphic to $\Hom_R(E_R \tensor_R R^{1 \over q}, E_R)$ which is precisely the Matlis dual of $E_R \tensor_R R^{1 \over q}$ as desired.  But then we see, by Matlis duality, that the map \ref{compositionSurjective} is surjective if and only if $d^{1 \over p^e} F^e : E_R \rightarrow E_R \tensor_R R^{1 \over q}$ is injective.

Instead of proving the equivalence of (ii) and (iii) we simply refer to the proof of \cite[Theorem 3.9]{TakagiInversion}.  In that proof, Takagi shows that $d^{1 \over p^e} F^e : E_R \rightarrow E_R \tensor_R R^{1 \over p^e}$ injects if and only if $d (I^{[p^e]} : I) \nsubseteq \bm^{[p^e]}$, which is precisely what we need.
\end{proof}

\begin{theorem}\cite[Theorem 3.1]{VassilevTestIdeals}
Let $S$ be an $F$-finite regular local ring and $I \subseteq R$ a radical ideal.  Suppose that $\ba \subseteq S$ is an ideal containing $I$ and $t > 0$ is a real number.  Let us denote $S/I$ by $R$ and  $\ba / I$ by $\overline{\ba}$.  If $\phi : S \rightarrow R$ is the canonical surjection, then $\ba^{\lceil t(q-1) \rceil} (I^{[q]} : I) \subseteq ((\phi^{-1}(\tau_R(\overline{\ba}^t)))^{[q]} : \phi^{-1}(\tau_R(\overline{\ba}^t)))$ for all $q = p^e > 0$.
\end{theorem}

\begin{proof}
The proof is a straightforward generalization of Vassilev's proof of the same statement in the case that $\ba = R$.
Given $J \subseteq S$ which contains $I$, let us denote $J/I \subseteq R$ by $\overline{J}$.  Notice that for all $\overline J \subseteq R$, all $q = p^e$, all $c \in \tau_R(\overline{\ba}^t)$ and all $z \in \overline{J}^{*\overline{\ba}^t}$ we have $c \overline{\ba}^{\lceil t (q-1) \rceil} z^q \subseteq \overline{J}^{[q]}$.
In other words, we have
\[
(\tau_R(\overline{\ba}^t)) \overline{\ba}^{\lceil t (q-1) \rceil} ( (\overline{J}^{*\ba^t})^{[q]}) \subseteq \overline{J}^{[q]}.
\]
Pulling this back to $S$,
we obtain
\[
\small
\xymatrix{
\phi^{-1}(\tau_R(\overline{\ba}^t)) \ba^{\lceil t (q-1) \rceil} (\phi^{-1}(\overline{J}^{*\overline{\ba}^t}))^{[q]} \subseteq \phi^{-1}(\tau_R(\overline{\ba}^t)) \phi^{-1}(\overline{\ba}^{\lceil t (q-1) \rceil}) (\phi^{-1}((\overline{J}^{*\overline{\ba}^t})^{[q]})) \subseteq J^{[q]} + I.
}
\]
Choose $w \in (I^{[q]} : I)$ and note that it is sufficient to show that $w \ba^{\lceil t(q-1) \rceil} \subseteq ((\phi^{-1}(\tau_R(\overline{\ba}^t)))^{[q]} : \phi^{-1}(\tau_R(\overline{\ba}^t)))$.  Observe that
\[
w \phi^{-1}(\tau_R(\overline{\ba}^t))  \ba^{\lceil t (q-1) \rceil} (\phi^{-1}(\overline{J}^{*\overline{\ba}^t}))^{[q]} \subseteq w J^{[q]} + I^{[q]} \subseteq J^{[q]}.
\]
Therefore,
\[
w \phi^{-1}(\tau_R(\overline{\ba}^t))  \ba^{\lceil t (q-1) \rceil} \subseteq ( J^{[q]} : (\phi^{-1}(\overline{J}^{*\overline{\ba}^t}))^{[q]} ) = (J : (\phi^{-1}(\overline{J}^{*\overline{\ba}^t})))^{[q]}.
\]
But this implies that $w \phi^{-1}(\tau_R(\overline{\ba}^t))  \ba^{\lceil t (q-1) \rceil} \subseteq \bigcap_{I \subseteq J} (J :(\phi^{-1}(\overline{J}^{*\overline{\ba}^t})))^{[q]}$.  By \cite[Lemma 2.1]{VassilevTestIdeals}, we see that $w \ba^{\lceil t (q-1) \rceil} \phi^{-1}(\tau_R(\overline{\ba}^t)) \subseteq \left(\bigcap_{I \subseteq J} (J : (\phi^{-1}(\overline{J}^{*\overline{\ba}^t})))\right)^{[q]}$.

Thus, we simply have to prove that $\bigcap_{I \subseteq J} (J : (\phi^{-1}(\overline{J}^{*\ba^t}))) = \phi^{-1}(\tau_R(\overline{\ba}^t))$, and we will be done.
Therefore, we consider
\[
v \in \bigcap_{I \subseteq J} (J : (\phi^{-1}(\overline{J}^{*\ba^t}))).
\]
This is equivalent to $v \phi^{-1}(\overline{J}^{*\ba^t}) \subseteq J$ for all $\overline{J} \subseteq R$.  Of course, this is the same as the condition $\phi(v) \overline{J}^{*\ba^t} \subseteq \overline{J}$, or in other words $\phi(v) \in (\overline{J} : \overline{J}^{*\ba^t})$ for all $J \subseteq R$.  This last statement is equivalent to $v \in \phi^{-1}(\tau_R(\overline{\ba}^t))$ as desired.
\end{proof}

\begin{corollary}
\label{GeneralizedVassilevSubadjunction}
Suppose $R$ is a quotient of an $F$-finite regular ring, $\ba \subseteq R$ an ideal such that $\ba \cap R^{\circ} \neq \emptyset$ and $t > 0$ a real number such that $(R, \ba^t)$ is \sfp. Then $R / \tau_R(\ba^t)$ is $F$-pure.
\end{corollary}

\section{An application to $F$-pure thresholds}
\label{SectionApplicationToFPureThresholds}

We begin by recalling the definition of the $F$-pure threshold.  As before, we assume that $R$ is an $F$-finite reduced ring and that $\ba \cap R^{\circ} \neq \emptyset$.

\begin{definition} \cite[Definition 2.1]{TakagiWatanabeFPureThresh} \cite[Proposition 2.2(3)]{TakagiWatanabeFPureThresh}
Let $R$ be a reduced $F$-finite ring and $\ba \subseteq R$ an ideal such that $\ba \cap R^{\circ} \neq \emptyset$.  We define the $F$-pure threshold, $\fpt(\ba)$ to be
\[
\fpt(\ba) = \sup\{s \in \bR_{\geq 0} | \text{ $(R, \ba^s)$ is $F$-pure} \} = \sup\{s \in \bR_{\geq 0} | \text{ $(R, \ba^s)$ is strongly $F$-pure} \}.
\]
\end{definition}

Our next goal is to show that the $F$-pure threshold can be described using \sfpty.  First we need a lemma.

\begin{lemma}
\label{SharplyFPureImpliesSmallerStronglyFPure}
If $(R, \ba^s)$ is sharply $F$-pure, then $(R, \ba^{s-\epsilon})$ is strongly $F$-pure (and thus also $F$-pure) for all $\epsilon$ satisfying $0 < \epsilon \leq s$.
\end{lemma}
\begin{proof}
Choose $e > 0$ that satisfies $\epsilon p^e > s$ and such that there exists an $f \in \ba^{\lceil s (p^e - 1) \rceil }$ so that the map $f^{1 \over p^e} F^e$ splits.  Then $s(p^e - 1) > p^e(s - \epsilon)$ and so  $\lceil s(p^e - 1) \rceil \geq \lceil p^e(s - \epsilon) \rceil$ which implies that
\[
f \in \ba^{\lceil s (p^e - 1) \rceil } \subseteq \ba^{\lceil (s - \epsilon)(p^e) \rceil}.
\]
But then we see that $(R, \ba^{s - \epsilon})$ is strongly $F$-pure as desired.
\end{proof}

\begin{proposition}
\label{FPureThresholdInTermsOfSharpPurity}
Let $R$ and $\ba$ be as above then,
\[
\fpt(\ba) = \sup\{s \in \bR_{\geq 0} | \text{ the pair $(R, \ba^s)$ is \sfp} \}.
\]
\end{proposition}
\begin{proof}
First note that a strongly $F$-pure pair is always sharply $F$-pure, so we see that the inequality $\leq$ is easy.  The other inequality is immediately implied by Lemma \ref{SharplyFPureImpliesSmallerStronglyFPure}.
\end{proof}

Using this, we can show that the $F$-pure threshold is a certain type of rational number under certain conditions; compare with \cite{BlickleMustataSmithDiscretenessAndRationalityOfFThresholds}, \cite{KatzmanLyubeznikZhangOnDiscretenessAndRationality}, or \cite{HaraMonskyFPureThresholdsAndFJumpingExponents}.

\begin{corollary}
\label{CorollaryFPureThresholdIsRational}
Suppose that $(R, \ba^{\fpt(\ba)})$ is a sharply $F$-pure pair.  Then $\fpt(\ba)$ is a rational number and, furthermore, there exists some $e > 0$ such that $\fpt(\ba)(p^e-1) \in \bZ$.  In other words, $\fpt(\ba)$ can be written as a quotient $m \over n$ where $m, n \in \bZ$ and $p$ does not divide $n$.
\end{corollary}
\begin{proof}
Assume that $\fpt(\ba)$ cannot be written as such a quotient.  Now suppose $(R, \ba^{\fpt(\ba)})$ is sharply $F$-pure, then there exists an $f \in \ba^{\lceil \fpt(\ba)(p^e - 1) \rceil}$ such that $f^{1 \over p^e} F^e : R \rightarrow R^{1 \over p^e}$ splits.  But then note that $f \in \ba^{\lceil (\fpt(\ba)+{\epsilon}) (p^e - 1) \rceil}$ for some sufficiently small $\epsilon > 0$, since $\fpt(\ba)(p^e - 1)$ is not an integer.  This contradicts Proposition \ref{FPureThresholdInTermsOfSharpPurity}.
\end{proof}

\begin{remark}
Mircea Musta{\c{t}}{\v{a}} has suggested the following argument which gives a partial converse to Corollary \ref{CorollaryFPureThresholdIsRational}.  Consider the situation where $(R, \bm)$ is an $F$-finite regular local ring and $\ba = (f)$ is principal.  Suppose that $\fpt(f)$ is a rational number such that $\fpt(f)(p^e - 1) \in \bZ$ for some $e > 0$ and $\fpt(f) < 1$. We will sketch Musta{\c{t}}{\v{a}}'s argument that $(R, f^{\fpt(f)})$ is (sharply) $F$-pure.

Following the notation of \cite{MustataTakagiWatanabeFThresholdsAndBernsteinSato}, we set $\nu_f(p^e) = \max \{s \in \bZ_{>0} | f^s \notin \bm^{[p^e]}\}$.  Fix an $e > 0$ and note that $p^d \nu_f(p^e) \leq \nu_f(p^{e+d}) \leq p^d(\nu_f(p^e) + 1)$ since $f^r \in \bm^{[p^e]}$ if and only if $f^{rp^d} \in \bm^{[p^{e+d}]}$.  Dividing through by $p^{d+e}$ gives us
\[
{ \nu_f(p^e) \over p^e} \leq {\nu_f(p^{e+d}) \over p^{e+d}} \leq { \nu_f(p^e) + 1 \over p^e}.
\]
We then take the limit as $d$ goes to infinity and obtain, by \cite[Remark 1.5]{MustataTakagiWatanabeFThresholdsAndBernsteinSato}, that
$ { \nu_f(p^e) \over p^e} \leq \fpt(f) \leq { \nu_f(p^e) + 1 \over p^e}$,
for every $e > 0$.  Now multiplying through by $p^e$ we get
\begin{equation}
\label{KeyEquationToRound}
\nu_f(p^{e}) \leq p^e \fpt(f) \leq \nu_f(p^e) + 1
\end{equation}
for every $e > 0$.  Note that both inequalities are actually strict since under our hypotheses, $p^e \fpt(f)$ is never an integer.

Fix an $e > 0$ such that $(p^e - 1)\fpt(f)$ is an integer.  To show that $(R, f^{\fpt(f)})$ sharply $F$-pure, it is sufficient to show that $\nu_f(p^e) = \lceil\fpt(f)(p^e-1) \rceil = \fpt(f)(p^e-1)$ for some $e > 0$ by Theorem \ref{SharpFeddersCriterion}.
First observe that the fractional part $\{ \fpt(f) p^e \} = \fpt(f) p^e - \lfloor \fpt(f) p^e \rfloor$ is equal to $\fpt(f)$ because $\{ \fpt(f) p^e \} = \{ \fpt(f) (p^e - 1) + \fpt(f) \} = \{ \fpt(f) \}$.  By rounding down equation \ref{KeyEquationToRound}, we see that
\[
\nu_f(p^e) = \lfloor \fpt(f) p^e \rfloor = \fpt(f) p^e - \fpt(f) = \fpt(f) (p^e - 1) < \nu_f(p^e) + 1,
\]
which completes the argument.
\end{remark}

We conclude with one more application of Lemma \ref{SharplyFPureImpliesSmallerStronglyFPure}.

\begin{corollary}
\label{SharpFPureTypeImpliesLC}
If $(R, \ba^t)$ has dense \sfp{ }type and $X = \Spec R$ is normal and $\bQ$-Gorenstein, then $(X, \ba^t)$ is log canonical.
\end{corollary}

For a definition of log canonical singularities see \cite{KollarMori}.  By \emph{dense \sfp{ }type} we simply mean to generalize the usual definition of \emph{dense $F$-pure type}; see, for example, \cite[Definition 1.9]{TakagiWatanabeFPureThresh}, by replacing the phrase ``$F$-pure'' wherever it occurs, with ``sharply $F$-pure''.  Roughly, \emph{dense \sfp{ }type} means that, after reducing the pair to a collection of positive characteristic models, a Zariski dense set of those models is sharply $F$-pure.  In order to prove Corollary \ref{SharpFPureTypeImpliesLC}, one could repeat the argument of \cite[Theorem 3.3]{HaraWatanabeFRegFPure} or \cite[Theorem 3.8]{TakagiInversion}.  We instead use Lemma \ref{SharplyFPureImpliesSmallerStronglyFPure} and apply \cite[Theorem 3.8]{TakagiInversion}.

\begin{proof}
It is easy to see that $(X, \ba^t)$ is log canonical if and only if $(X, \ba^{t - \epsilon})$ is log canonical for every $\epsilon$ satisfying $t > \epsilon > 0$.  But note that if $(R, \ba^t)$ has dense \sfp{ }type, then $(R, \ba^{t - \epsilon})$ has dense $F$-pure type by Proposition \ref{SharplyFPureImpliesSmallerStronglyFPure}.  But then by \cite[Theorem 3.8]{TakagiInversion}, $(X, \ba^{t - \epsilon})$ is log canonical.
\end{proof}

\begin{remark}
Suppose that $(R, \ba^t)$ has dense $F$-pure type and $t = a/b$ is a rational number where $a$ and $b$ are integers.  Further suppose that $p$ is a prime that doesn't appear as a factor of $b$ and note that all but a finite number of primes satisfy this property.  Therefore, for infinitely many $e > 0$, $t(p^e - 1)$ is an integer.  It thus follows that $(R, \ba^t)$ has dense \sfp{ }type.
\end{remark}

\bibliographystyle{skalpha}
\bibliography{ThesisBib}

\end{document}